\documentclass[11pt,twoside]{amsart}
\usepackage{graphicx}
\usepackage{color}
\usepackage{amssymb}
\usepackage{amsmath}
\usepackage{amsthm}
\usepackage{amsfonts}
\usepackage{amsmath, amsfonts, amssymb}
\newtheorem{theorem}{Theorem}[section]

\newtheorem{de}{Definition}
\newtheorem{lm}{Lemma}[section]

\newtheorem{theor}{Theorem}

\newtheorem{lemma}[theorem]{Lemma}

\numberwithin{equation}{section}

\begin{document}
	
\title{Classification of Frobenius algebra structures on two-dimensional vector space over any base field}
		\author{D.Asrorov$^1$, U.Bekbaev$^2$, I.Rakhimov$^3$}
		
		\thanks{{\scriptsize
				emails: $^1$96asrorovdiorjon@mail.ru; $^2$uralbekbaev@gmail.com; $^3$isamiddin@uitm.edu.my}}
		\maketitle
		\begin{center}
			\address{$^1$Samarkand State University, Samarkand, Uzbekistan\\ $^2$Turin Polytechnic University in Tashkent, Tashkent, Uzbekistan\\ $^3$Universiti Teknologi MARA (UiTM), Shah Alam, Malaysia }
			\end{center}

			\begin{abstract} Classifying Frobenius algebras is a key question that has been addressed in various contexts. The structure of finite-dimensional Frobenius algebras depends on the base field and the dimension of the algebra, leading to different classification results depending on whether the base field has characteristic zero, characteristic $p$, or other properties. Frobenius algebras over fields of characteristic zero have been well-studied, often related to semisimple algebra theory.
The behavior of Frobenius algebras over fields of positive characteristic presents new challenges, with connections to modular representation theory. In the paper, we first classify all associative algebra structures on a two-dimensional vector space over any base field equipped with a non-degenerate bilinear form. We then identify which of these are Frobenius algebras. Lists of canonical representatives of the isomorphism classes of these algebras are provided for a base field with characteristics not equal to two or three, as well as for characteristics two and three.
\end{abstract}
	

\textbf{Mathematics Subject Classification (MSC)}: Primary 16W20; 16S34; Secondary 20C05

	\section{Introduction}
Frobenius algebras are a fundamental concept in mathematics, appearing in areas such as algebra, topology, and theoretical physics, particularly in the study of categories, representation theory, and quantum field theory. They are named after the German mathematician Ferdinand Georg Frobenius.

%
A Frobenius algebra is a vector space $\mathbb{V}$ equipped with:
\begin{itemize}
\item An associative unital algebra structure $\mathbb{A}=(\mathbb{V},\mu,1),$ where $\mu:\mathbb{V}\times \mathbb{V} \longrightarrow \mathbb{V}$ represents the multiplication and $1$ is the unit element.
\item  A nondegenerate bilinear form $\sigma:\mathbb{V}\times \mathbb{V}\longrightarrow \mathbb{F}$, which satisfies the Frobenius property, namely:
$$\sigma(\mu(\mathrm{x, y}),\mathrm{z})=\sigma(\mathrm{x},\mu(\mathrm{y, z}))\ \mbox{for all}\ \mathrm{x,y,z} \in \mathbb{V}.$$
\end{itemize}

Further we use the juxtaposition $\mathrm{x y}$ notation for the product $\mu(\mathrm{x, y}).$  The bilinear form $\sigma$ is called the Frobenius form of the algebra $\mathbb{A}$.

\textbf{Key Properties}:
\begin{itemize}
\item \textbf{Associativity}. The algebra multiplication must satisfy:
$\mathrm{x(yz)=(xy)z}$.
\item \textbf{Nondegeneracy}. If
$\sigma(\mathrm{x,y})=0 $ for all $\mathrm{y} \in \mathbb{V}$ then $\mathrm{x}=0.$
\item \textbf{The Frobenius Property}. The compatibility condition can be written as:
$$\sigma(\mathrm{xy,z})=\sigma(\mathrm{x,yz})\ \mbox{for all}\ \mathrm{x,y,z} \in \mathbb{V}.$$
%
\item \textbf{Canonical Isomorphisms}. The existence of the non-degenerate bilinear form allows one to map between
$\mathbb{A}$ and its dual $\mathbb{A}^*=\mathrm{Hom}(\mathbb{A},\mathbb{F}).$
This is used in various constructions, such as representations and category theory.
\end{itemize}
\textbf{Examples:}
\begin{itemize}
\item \textbf{Group Algebras}. Let $G$ be a finite group and $\mathbb{F}[G]$  its group algebra over a field $\mathbb{F}$.
\begin{itemize}
\item The standard bilinear form defined by:
$\sigma(g,h)=\delta_{g,h^{-1}}$ (where $\delta$ is the Kronecker delta) gives
 a Frobenius algebra structure on $\mathbb{F}[G]$.
 \item The bilinear form $\sigma(f,g)$ given by the coefficient of the identity element in $f \cdot g$ equips $\mathbb{F}[G]$ by a Frobenius algebra structure.
  \end{itemize}
\item \textbf{Trace Functional}.
\begin{itemize}
\item Let $\mathbb{C}[u]$ be the polynomial ring and
$\mathbb{A}=\mathbb{C}[u]/(u^n)$ its quotient. A standard trace functional can be used to give this algebra the structure of a Frobenius algebra.
\item Any matrix algebra defined over a field $\mathbb{F}$ is a Frobenius algebra with Frobenius form $\sigma(A,B)=\mathrm{tr}(A \cdot B)$, where $\mathrm{tr}(X)$ is the trace of matrix $X$.
\item Any finite-dimensional unital associative algebra $\mathbb{A}$ has a natural homomorphism to its own endomorphism ring $\mathrm{End}(\mathbb{A})$. A bilinear form can be defined on $\mathbb{A}$ in the sense of the previous example. If this bilinear form is non-degenerate, then it equips $\mathbb{A}$ with the structure of a Frobenius algebra.
 \end{itemize}
%
%
%
%
%
\end{itemize}
%

Frobenius algebras began to be studied in the 1930s by R. Brauer and C. Nesbitt \cite{BrauerNesbitt}. T. Nakayama discovered the beginnings of a rich duality theory \cite{Nakayama1, Nakayama2}. J. Dieudonn\'{e} used this to characterize Frobenius algebras \cite{Dieudonné(1958)}. Frobenius algebras were generalized to quasi-Frobenius rings, those Noetherian rings whose right regular representation is injective. In recent times, interest has been renewed in Frobenius algebras due to connections to Topological Quantum Field Theory. TQFTs are functors from the category of bordisms to the category of vector spaces. It has been found that they play an important role in the algebraic treatment and axiomatic foundation of Topological Quantum Field Theory \cite{LA, TQFT1, Lauda0, Lauda, TQFT2}. Frobenius algebras underlie the algebraic structure of 2D Topological Quantum Field Theory's (TQFT's). They provide a bridge between physics and algebraic topology by encoding information about 2-dimensional surfaces and their invariants. Let us mention a few results, illustrating the importance of the concept. In \cite{DR} the author introduces foundational concepts related to Frobenius algebras in the context of Hopf algebra theory (also see \cite{paul}). M. Atiyah in \cite{A}  discussed the role of Frobenius algebras in the development of TQFT's and first described their axiomatic foundation. The authors of \cite{CIM} present a unified approach to the study of separable and Frobenius algebras.

In the paper, we first classify all associative algebra structures on a two-dimensional vector space over any base field equipped with a non-degenerate bilinear form (Section 3). Then identify which of these are Frobenius algebras (Section 4).

\section{Preliminaries}
Let $\mathbb{A}$ be a PI-algebra, with a given set of polynomial identities $\{P_j[u_1,u_2,...,u_n]=0:\ j\in J\}$, over a field $\mathbb{F}$. Let $P_j[u_1,u_2,...,u_n]=\sum_{i=1}^{k_j}Q^i_j[u_1,u_2,...,u_n] R^i_j[u_1,u_2,...,u_n]$, $j\in J$. In some applications this kind of algebras appear with a non-degenerate bilinear form $\sigma:\mathbb{A}\times \mathbb{A}\rightarrow \mathbb{F}$ such that $\sum_{i=1}^{k_j}\sigma(Q^i_j[a_1,a_2,...,a_n], R^i_j[a_1,a_2,...,a_n])=0$ at all $a_1,a_2,...,a_n\in\mathbb{A}.$ A pair $(\mathbb{A}, \sigma)$ is said to be a Frobenius PI-algebra. The classification of a given class of Frobenius PI-algebras is of great interest. In this paper we consider as a class PI-algebras the class of associative algebras and provide a complete classification of Frobenius PI-algebras on two-dimensional vector space over any base field.

The next theorem establishes the equivalence of several important and useful characterizations of Frobenius algebras (see \cite{CR}).
\begin{theor} The following statements about a
finite-dimensional algebra $\mathbb{A}$ are equivalent:
\begin{itemize}
\item $\mathbb{A}$ is a Frobenius algebra.
\item There exists a non-degenerate bilinear form, $\sigma: \mathbb{A}\times \mathbb{A} \longrightarrow \mathbb{F}$ such that
$\sigma(\mathrm{xy, z}) =\sigma(\mathrm{x, yz})$ for all $\mathrm{x,y,z} \in \mathbb{A}$.
\item There exists a linear function $\lambda \in \mathbb{A}^*$ whose kernel contains no left or right ideals different from zero.
\item For all left ideals $L$ and right ideals $R$ in $\mathbb{A}$ we have
$$l(r(L)) = L,\ \mbox{and}\ (r(L):\mathbb{F}) + (L:\mathbb{F}) =(\mathbb{A}:\mathbb{F});$$
$$r(l(R)) =R, \ \mbox{and}\ (l(R): \mathbb{F}) + (R: \mathbb{F}) =(\mathbb{A}: \mathbb{F}),$$
where $r(L)=\{\mathrm{x} \in \mathbb{A}: L\mathrm{x}=0\}$ and $l(R)=\{\mathrm{x} \in \mathbb{A}: \mathrm{x}R=0\}$ are right and left annihilators, respectively.
\end{itemize}
\end{theor}
\begin{de} Let $(\mathbb{A},\sigma)$ and $(\mathbb{B},\tau)$ be Frobenius algebras on a vector space $\mathbb{V}$ over a field $\mathbb{F}$ with non-degenerate bilinear forms $\sigma:\mathbb{A}\times \mathbb{A}\rightarrow \mathbb{F}$ and $\tau:\mathbb{B}\times \mathbb{B}\rightarrow \mathbb{F}$, respectively. Pairs $(\mathbb{A},\sigma)$ and $(\mathbb{B},\tau)$ are said to be isomorphic if there exists isomorphism of algebras $f:\mathbb{A}\longrightarrow \mathbb{B}$  such that
$\sigma(\mathrm{xy})=\tau(f(\mathrm{x})f(\mathrm{y}))$ for all $\mathrm{x,y}\in \mathbb{A}$.
\end{de}

Let $\mathbb{A}$ be an $n$-dimensional algebra and $\mathbf{e}=(e_1,e_2,...,e_n)$ its basis. Then $\mathrm{x}$ and $\mathrm{y}$ can be presented by their coordinate vectors $x^T=(x_1,x_2,...,x_n)$ and $y^T=(y_1,y_2,...,y_n)$ as
$$\mathrm{x}=\mathbf{e}x^T, \mathrm{y}=\mathbf{e}y^T\ \mbox{and}\ \mathrm{z}=\mathbf{e}z^T, \ \mbox{respectevely}.$$
Therefore, $\mathrm{xy}=\mathbf{e}A(x\otimes y),$ where
$$A=\left(\begin{array}{ccccccccccccc}a_{11}^1&a_{12}^1&...&a_{1n}^1&a_{21}^1&a_{22}^1&...&a_{2n}^1&...&a_{n1}^1&a_{n2}^1&...&a_{nn}^1\\ a_{11}^2&a_{12}^2&...&a_{1n}^2&a_{21}^2&a_{22}^2&...&a_{2n}^2&...&a_{n1}^2&a_{n2}^2&...&a_{nn}^2 \\
...&...&...&...&...&...&...&...&...&...&...&...&...\\ a_{11}^n&a_{12}^n&...&a_{1n}^n&a_{21}^n&a_{22}^n&...&a_{2n}^n&...&a_{n1}^n&a_{n2}^n&...&a_{nn}^n\end{array}\right)$$ defined by:
$$\mathrm{e}_i  \mathrm{e}_j=\sum\limits_{k=1}^n a_{ij}^k\mathrm{e}_k, \ \mbox{where}\ i,j=1,2,...,n.$$
The matrix $A$ is said to be the matrix of structure constants (MSC) of $\mathbb{A}$ on the basis $\mathbf{e}$.

If $\mathbb{A}$ is a Frobenius algebra then the Frobenius map $\sigma$ also is presented by its matrix $S$: $\sigma(\mathrm{x,y})=x^TSy$. Then
$$\sigma\mathrm{(xy,z})=(x^T\otimes y^T)A^TSz \ \mbox{and}\ \sigma\mathrm{(x,yz})=x^TSA(y\otimes z).$$ Therefore, one has
\begin{lm} \label{FA}
An algebra $\mathbb{A}$ is Frobenius if and only if \begin{equation}(x^T\otimes y^T)A^TSz=x^TSA(y\otimes z).\end{equation}
\end{lm}

 In \cite{Rakhimov} a result on the classification of two-dimensional associative algebras, over any base field $\mathbb{F}$ and their automorphism groups was given. The result was based on technique applied in \cite{UB}, we keep here the notations of \cite{UB} as the following theorems for the characteristic of $\mathbb{F}$ to be not two and three, to be two and to be three, respectively.
 	 	 	\begin{theor} \label{Char0} Any non-trivial $2$-dimensional associative algebra over a field $\mathbb{F}$ $(Char(\mathbb{F})\neq 2,3)$ is isomorphic to only one of the following listed, by their matrices of structure constants, algebras:
 		\begin{itemize}
 			\item $A_3(1,0,0)=\begin{pmatrix}
 			1&0&0&0 \\
 			0&0 &0& 0
 			\end{pmatrix},$
 			\item $A_3(1,0,1)=\begin{pmatrix}
 				1&0&0&0 \\
 				0&1&0& 0
 			\end{pmatrix},$
 			\item $A_3(\frac{1}{2},0,0)=\begin{pmatrix}
 			\frac{1}{2}&0&0&0 \\
 			0&0 &\frac{1}{2}& 0
 			\end{pmatrix},$ 			 			
 			\item $A_3(\frac{1}{2},\alpha_4,\frac{1}{2})=\begin{pmatrix}
 			\frac{1}{2}&0&0&\alpha_4  \\
 			0&\frac{1}{2}&\frac{1}{2}& 0
 			\end{pmatrix}\simeq\begin{pmatrix}
 			\frac{1}{2}&0&0&r^2\alpha_4 \\
 			0&\frac{1}{2}& \frac{1}{2}& 0
 			\end{pmatrix},$\ where  $ \alpha_4, r\in\mathbb{F}$, $r\neq 0$,
 			 			\item $A_{13}=\begin{pmatrix}
 			0&0&0&0 \\
 			1&0&0&0
 			\end{pmatrix}.
 			$
 		\end{itemize}
 		
 		 Among the listed above algebras only $A_3(\frac{1}{2},\alpha_4,\frac{1}{2})$ is unital and its unit element is $\mathbf{1}=\begin{pmatrix}
 		2\\
 		0
 	\end{pmatrix}.$
 	\end{theor}

\begin{theor} \label{Char2} Any non-trivial $2$-dimensional associative algebra over a field $\mathbb{F},$ $(Char(\mathbb{F})=2)$ is isomorphic to only one of the following listed by their matrices of structure constants, such algebras:

\begin{itemize}
\item $A_{3,2}(1,0,0)=$ $\left(\begin{array}{ccccc} 1&0&0&0\\
0&0&0&0
\end{array}\right),$
\item $A_{3,2}(1,0,1)=$ $\left(\begin{array}{ccccc} 1&0&0&0\\
0&1&0&0
\end{array}\right),$
\item  $A_{4,2}(1,\beta_1,0)=\left(\begin{array}{ccccc} 1&1&1&0\\
\beta_1&0&0&1
\end{array}\right)\ \cong \ \left(\begin{array}{ccccc} 1&1&1&0\\\beta_1+r+r^2&0&0&1
\end{array}\right),\ \mbox{where}\ r, \beta_1 \in \mathbb{F},$
\item
$A_{6,2}(1,0)=\left(\begin{array}{ccccc} 1&0&0&0\\
0&0&1&0
\end{array}\right),$
\item
$A_{11,2}(\beta_1)=\left(\begin{array}{ccccc} 0&1&1&0\\
\beta_1&0&0&1
\end{array}\right)\ \cong \ \left(\begin{array}{ccccc} 0&1&1&0\\
q^2\beta_1+r^2&0&0&1
\end{array}\right), \ \mbox{where}\ \beta_1, r, q \in \mathbb{F},  q\neq 0,$
\item $A_{12,2}=
\left(\begin{array}{ccccc} 0&0&0&0\\
1&0&0&0
\end{array}\right).$
\end{itemize}
Among the algebras listed above only
$A_{4,2}(1,\beta_1,0)$ and  $A_{11,2}(\beta_1)$ are
unital and theirs unit element is ${\bf 1}=\begin{pmatrix}
0\\
1
\end{pmatrix}.$

\end{theor}

\begin{theor} \label{Char3} Any non-trivial $2$-dimensional associative algebra over a field $\mathbb{F},$ $(Char(\mathbb{F})=3)$ is isomorphic to only one of the following listed by their matrices of structure constants, such algebras:

\begin{itemize}
\item
    $A_{3,3}(1,0,0)=\left(\begin{array}{ccccc} 1&0&0&0\\
0&0&0&0
\end{array}\right),$
\item
    $A_{3,3}(1,0,1)=\left(\begin{array}{ccccc} 1&0&0&0\\
0&1&0&0
\end{array}\right),$
\item
    $A_{3,3}(2,0,0)=\left(\begin{array}{ccccc} 2&0&0&0\\
0&0&2&0
\end{array}\right),$
\item $A_{3,3}(2,\alpha_4,2):=\left(\begin{array}{ccccc} 2&0&0&\alpha_4\\
0&2&2&0
\end{array}\right)\simeq \begin{pmatrix}
 		2&0&0&r^2\alpha_4 \\
 		0& 2& 2& 0
 		\end{pmatrix},$ where $ \alpha_4, r \in\mathbb{F}, r\neq 0$,
 		\item $A_{13,3}=\left(\begin{array}{ccccc} 0&0&0&0\\
 		1&0&0&0
 		\end{array}\right)$.
 		
\end{itemize}
Here only $A_{3,3}(2,\alpha_4,2)$ is unital and its unit element is ${\bf 1}=\begin{pmatrix}
2\\
0
\end{pmatrix}.$

\end{theor}	

\begin{theor} \label{AutChar0} The automorphism groups of the algebras listed in Theorem \ref{Char0} are given as follows\\	
	$Aut(A_3(1,0,0))=\left\{\begin{pmatrix}
	1&0\\ 0&t
	\end{pmatrix}:\ 0\neq t\in \mathbb{F}\right\},$\\	
	$Aut(A_3(1,0,1))=\left\{\begin{pmatrix}
	1&0\\ s&t
	\end{pmatrix}:\ s, t\in \mathbb{F},\ t\neq 0\right\},$\\	
	$Aut(A_3(\frac{1}{2},0,0))=\left\{\begin{pmatrix}
	1&0\\ s&t
	\end{pmatrix}:\ s, t\in \mathbb{F},\ t\neq 0\right\},$\\	
	$Aut(A_3(\frac{1}{2},0,\frac{1}{2}))=\left\{\begin{pmatrix}
	1&0\\ 0&t
	\end{pmatrix}: \ t\neq 0\right\}$, \\	
	$Aut(A_3(\frac{1}{2},\alpha_4,\frac{1}{2}))= \left\{\begin{pmatrix}
	1&0\\ 0&\pm1
	\end{pmatrix}\right\}$, whenever $\alpha_4\neq 0$,\\	
	$Aut(A_{13})=\left\{\begin{pmatrix}
	p&0\\ s&p^2
	\end{pmatrix}:\ p, s\in \mathbb{F}, p\neq 0\right\}.$\end{theor}

\begin{theor}  \label{AutChar2} The automorphism groups of the algebras listed in Theorem \ref{Char2} are given as follows\\		
$Aut(A_{3,2}(1,0,0))=\left\{ \left(\begin{array}{crr} 1 & 0\\ 0 &t \end{array}\right): \ 0 \neq t \in \mathbb{F}\right\},$\\
$Aut(A_{3,2}(1,0,1))=\left\{ \left(\begin{array}{crr} 1 & 0\\ s & t \end{array}\right): \ s, t \in \mathbb{F},\ t \neq 0\right\},$\\
$Aut(A_{4,2}(1,0,0))=\left\{ \left(\begin{array}{crr} p & 0\\ s & 1 \end{array}\right), \ p \neq 0, \ s \in \mathbb{F} \right\},$\\
$Aut(A_{4,2}(1,\beta_1,0))=\left\{ \left(\begin{array}{crr} 1 & 0\\ s & 1 \end{array}\right): \ s \in \mathbb{F} \right\},$ \ whenever $\beta_1\neq 0$,\\
$Aut(A_{6,2}(1,0))=\left\{ \left(\begin{array}{crr} 1 & 0\\ 0 & 1 \end{array}\right)\right\},$\\
	$Aut(A_{11,2}(\beta_1))=\left\{ \left(\begin{array}{crr} p & 0\\ \beta_1(p-1) & 1 \end{array}\right): \  \beta_1,p \in \mathbb{F}, \ p \neq 0 \right\}$,\\
$Aut(A_{12,2})=\left\{ \left(\begin{array}{lrr} p & 0\\ s & p^2 \end{array}\right): \  p, s\in \mathbb{F}, p\neq 0 \right\}.$

\end{theor}

\begin{theor}  \label{AutChar3} The automorphism groups of the algebras listed in Theorem \ref{Char3} are given as follows\\
$Aut(A_{3,3}(1,0,1))=\left\{ \left(\begin{array}{crr} 1 & 0\\ s & t \end{array}\right):\  s, t\in \mathbb{F}, t \neq 0 \right\},$\\
$Aut(A_{3,3}(1,0,0))=\left\{ \left(\begin{array}{crr} 1 & 0\\ 0 & t \end{array}\right):\ 0\neq t\in \mathbb{F} \right\},$\\
$Aut(A_{3,3}(2,0,0))=\left\{ \left(\begin{array}{crr} 1 & 0\\ 1+2t & t \end{array}\right):\ 0 \neq t\in \mathbb{F }\right\},$\\
 $Aut(A_{3,3}(2,0,2))=\left\{ \left(\begin{array}{crr} 1 & 0\\ 0 & 1 \end{array}\right)\right\},$\\
$Aut(A_{3,3}(2,\alpha_4,2))=\left\{ \left(\begin{array}{crr} 1 & 0\\ s & 1 \end{array}\right):\ s \in \mathbb{F} \right\},$ \ provided that $\alpha_4\neq0$,\\
$Aut(A_{13,3})=\left\{ \left(\begin{array}{crr} p & 0\\ s & 2p^2 \end{array}\right):\ p, s\in \mathbb{F}, p\neq 0 \right\}.$
\end{theor}

Let $\mathbb{A}$ be a two-dimensional algebra over a field $\mathbb{F}$, $\mathbf{e}=\left(\mathrm{e}_1, \mathrm{e}_2\right)$ its basis, $A=\begin{pmatrix}
 				\alpha_1&\alpha_2&\alpha_3&\alpha_4 \\
 				\beta_1&\beta_2&\beta_3& \beta_4
 			\end{pmatrix}$ MSC of $\mathbb{A}$ and $S=\begin{pmatrix}
 a & b\\
 c & d
 \end{pmatrix}\in GL(2,\mathbb{F})$ the matrix of $\sigma$. Then
\begin{equation} \label{Form}
\begin{array}{lllllllllllllll}
(x^T\otimes y^T)A^TSz&=&(\alpha_1a+\beta_1c)x_1y_1z_1+(\alpha_2a+\beta_2c)x_1y_2z_1+(\alpha_3a+\beta_3c)x_2y_1z_1\\
&&+(\alpha_4a+\beta_4c)x_2y_2z_1+ (\alpha_1b+\beta_1d)x_1y_1z_2+(\alpha_2b+\beta_2d)x_1y_2z_2\\
&&+(\alpha_3b+\beta_3d)x_2y_1z_2+(\alpha_4b+\beta_4d)x_2y_2z_2\\

x^TSA(y\otimes z)&=&(\alpha_1a+\beta_1b)x_1y_1z_1+(\alpha_3a+\beta_3b)x_1y_2z_1+(\alpha_1c+\beta_1d)x_2y_1z_1\\
&&+(\alpha_3c+\beta_3d)x_2y_2z_1+ (\alpha_2a+\beta_2b)x_1y_1z_2\\
&&+(\alpha_4a+\beta_4b)x_1y_2z_2+(\alpha_2c+\beta_2d)x_2y_1z_2+(\alpha_4c+\beta_4d)x_2y_2z_2\\
\end{array}
\end{equation}  and (\ref{FA}) can be written as follows
\begin{equation} \label{Form2}
 \left\lbrace
 \begin{array}{ccc}
\alpha_1a+\beta_1c-\alpha_1a-\beta_1b&=&0\\
\alpha_2a+\beta_2c-\alpha_3a-\beta_3b&=&0\\
\alpha_3a+\beta_3c-\alpha_1c-\beta_1d&=&0\\
\alpha_4a+\beta_4c-\alpha_3c-\beta_3d&=&0\\
\alpha_1b+\beta_1d-\alpha_2a-\beta_2b&=&0\\
\alpha_2b+\beta_2d-\alpha_4a-\beta_4b&=&0\\
\alpha_3b+\beta_3d-\alpha_2c-\beta_2d&=&0\\
\alpha_4b+\beta_4d-\alpha_4c-\beta_4d&=&0\\
\end{array}
\right.
\end{equation} as far as  the system of monomial functions $\{x_1y_1z_1,x_1y_2z_1,x_2y_1z_1,x_1y_1z_2,x_2y_2z_1,x_1y_2z_2,x_2y_1z_2,x_2y_2z_2\}$
is linearly independent over $\mathbb{F}$.

\section{Classification of two-dimensional associative algebras equipped by a non-degenerate form $\sigma$}
In this section first we classify non-degenerate bilinear forms
given as $S=\begin{pmatrix}
	a &b\\ c &d
\end{pmatrix}$ with respect to the transformations $g^TS g$, $g\in G$, where $G$ is a fixed nontrivial automorphism's group from Theorems  \ref{AutChar0} -  \ref{AutChar3}. For the further usage the list of all nontrivial autmorphism's  groups in Theorems \ref{AutChar0} -  \ref{AutChar3} we enumerate as follows\\
 $G_1=\left\{\begin{pmatrix}
 1&0\\ 0&t
 \end{pmatrix}:\ 0\neq t\in \mathbb{F}\right\},$\	
 $G_2=\left\{\begin{pmatrix}
 1&0\\ s&t
 \end{pmatrix}:\ s,t\in \mathbb{F},\ t\neq 0\right\},$\	
 $G_3= \left\{\begin{pmatrix}
 1&0\\ s&1
 \end{pmatrix}:\ s\in \mathbb{F}\right\}$,\\	
 $G_4=\left\{\begin{pmatrix}
 p&0\\ s&p^2
 \end{pmatrix}:\ p,s\in \mathbb{F}, p\neq 0\right\}$($Char(\mathbb{F})\neq 3$),\\
$G_{5}=\left\{\begin{pmatrix}
p&0\\ s&2p^2
\end{pmatrix}:\ p,s\in \mathbb{F}, p\neq 0\right\}$($Char(\mathbb{F})=3$),\\
$G_{6, \beta_1}=\left\{\begin{pmatrix}
p&0\\ \beta_1(p-1)&1
\end{pmatrix}:\ \beta_1,p\in \mathbb{F}, p\neq 0\right\}$ ($Char(\mathbb{F})=2$),\\
 $G_{7}=\left\{\begin{pmatrix}
 1&0\\1+2t&t
 \end{pmatrix}:\ t\in \mathbb{F} , t\neq0\right\}$($Char(\mathbb{F})=3$), \\ $G_{8}=\left\{\begin{pmatrix}
1&0\\ 0&\pm 1
\end{pmatrix}\right\}$($Char(\mathbb{F})\neq 2$).

Now we treat the action $g^tS g$ for each $G_i, i=1,2,...,8$ $(g \in G_i)$ one by one.

Let $g=\begin{pmatrix}
1&0\\ 0&t
\end{pmatrix}\in G_1$. Then
$g^t\begin{pmatrix}
a &b\\ c &d
\end{pmatrix}g=\begin{pmatrix}
a &tb\\ tc&t^2d
\end{pmatrix}$ and the following canonical forms occur:
\begin{itemize}
\item  $ \begin{pmatrix}
a &1\\ c &d
\end{pmatrix}$, where $a,c,d\in \mathbb{F}$, $ad-c\neq 0$,\
 \item $\begin{pmatrix}
a &0\\ 1 &d
\end{pmatrix}$, where $a,d\in \mathbb{F}$, $ad\neq 0$,\
\item $\begin{pmatrix}
a &0\\ 0 &d
\end{pmatrix})\simeq \begin{pmatrix}
a &0\\ 0 &t^2d
\end{pmatrix}$, where $a,d,t\in \mathbb{F}, t \neq 0 ,ad\neq 0$.
\end{itemize}
If $g=\begin{pmatrix}
1&0\\ s&t
\end{pmatrix}\in G_2$ then $g^t\begin{pmatrix}
a &b\\ c &d
\end{pmatrix}g=\begin{pmatrix}
a+sc+sb+s^2d &tb+std\\ tc+std &t^2d
\end{pmatrix}$ and one comes to the following canonical forms:
\begin{itemize}
\item $\begin{pmatrix}
a&0\\ 1 &d
\end{pmatrix}$, where $a,d\in \mathbb{F}$, $ad\neq 0$,\
\item $ \begin{pmatrix}
a &0\\ 0 & d
\end{pmatrix}\simeq\begin{pmatrix}
a &0\\ 0 &t^2d
\end{pmatrix}$, where $a,d,t\in \mathbb{F}, t \neq 0, ad\neq 0$,\
\item $\begin{pmatrix}
0 &1\\ c &0
\end{pmatrix}$, where $c\in \mathbb{F}$, $c\neq 0, c +1\neq 0$,\
\item $\begin{pmatrix}
a &1\\ -1 &0
\end{pmatrix}$, where $a\in \mathbb{F}.$
\end{itemize}
Let $g=\begin{pmatrix}
1&0\\ s&1
\end{pmatrix}\in G_{3}$. Then
$g^t\begin{pmatrix}
a &b\\ c &d
\end{pmatrix}g=\begin{pmatrix}
a+sc+sb+s^2d &b+sd\\ c+sd &d
\end{pmatrix}$ and we get the following canonical forms:
\begin{itemize}
\item $\begin{pmatrix}
a&0\\ c &d
\end{pmatrix}$, where $a,c,d\in \mathbb{F}$, $ad\neq 0$,
\item $\begin{pmatrix}
0 &b\\ c &0
\end{pmatrix}$, where $b,c\in \mathbb{F}$, $bc\neq 0$, $b+c\neq 0,$
\item $\begin{pmatrix}
a &b\\ -b &0
\end{pmatrix}$, where $a,b\in \mathbb{F}$, $b\neq 0$.
\end{itemize}

If $g=\begin{pmatrix}
p&0\\ s&p^2
\end{pmatrix}\in G_4$, where $p,s\in \mathbb{F}, \ p\neq 0,$ then
$g^t\begin{pmatrix}
a &b\\ c &d
\end{pmatrix}g=\begin{pmatrix}
p^2a+psc+psb+s^2d&p^3b+p^2sd\\ p^3c+p^2sd &p^4d
\end{pmatrix}$ and the following canonical forms occur:
\begin{itemize}
\item $\begin{pmatrix}
a &0\\ c &d
\end{pmatrix}\simeq\begin{pmatrix}
p^2a &0\\ p^3c &p^4d
\end{pmatrix}, $ where $a,c,d,p\in \mathbb{F}$, $ad \neq 0$, $p\neq 0$,\
\item $\begin{pmatrix}
0 &b\\ c &0
\end{pmatrix}\simeq\begin{pmatrix}
	0 &p^3b\\ p^3c &0
\end{pmatrix} $, where $b,c,p\in \mathbb{F}$,$bc \neq 0$, $b \neq -c$, $p\neq 0$,\
\item $\begin{pmatrix}
a &b\\ -b&0
\end{pmatrix})\simeq\begin{pmatrix}
p^2a &p^3b\\ -p^3b &0
\end{pmatrix}$, where $a,b,p\in \mathbb{F}$, $b^2 \neq 0$, $p\neq 0.$
\end{itemize}

Let now  $g=\begin{pmatrix}
p&0\\ s&2p^2
\end{pmatrix}\in G_5$, where $p,s\in \mathbb{F},\ p\neq 0$ and $Char(\mathbb{F})=3$. Then we have
$$g^t\begin{pmatrix}
a &b\\ c &d
\end{pmatrix}g=\begin{pmatrix}
p^2a +psc+psb+s^2d&2p^3b+2p^2sd\\ 2p^3c+2p^2sd &p^4d
\end{pmatrix}$$ and the canonical forms are given as follows:
\begin{itemize}
\item $\begin{pmatrix}
a &0\\ c &d
\end{pmatrix}\simeq \begin{pmatrix}
p^2a &0\\ p^3c &p^4d
\end{pmatrix}, $ where $a,c,d,p\in \mathbb{F}$, $ad\neq 0$, $p \neq 0$,\
\item $\begin{pmatrix}
0 &b\\ c &0
\end{pmatrix}\simeq \begin{pmatrix}
0 &p^3b\\ p^3c &0
\end{pmatrix}$, where $b,c,p\in \mathbb{F}$, $bc\neq 0$, $b \neq -c$, $p \neq 0$,\
\item $\begin{pmatrix}
a &b\\ -b &0
\end{pmatrix}\simeq\begin{pmatrix}
p^2a &p^3b\\ -p^3b &0
\end{pmatrix}$, where $a,b,p\in \mathbb{F}$, $b\neq 0$, $p \neq 0$.
\end{itemize}

If $g=\left\{ \left(\begin{array}{crr} p & 0\\ \beta_1(p-1) & 1 \end{array}\right): \  \beta_1,p \in \mathbb{F}, p \neq 0\right\}\in G_{6,\beta_1}$($Char(\mathbb{F})=2$) then
$$g^t\begin{pmatrix}
a &b\\ c &d
\end{pmatrix}g=\begin{pmatrix}
p^2a+\beta_1(p-1)(pc+pb+\beta_1d(p-1)) &pb+\beta_1d(p-1)\\ pc+\beta_1d(p-1) &d
\end{pmatrix}$$ and one gets the following canonical forms:
\begin{itemize}
\item $\begin{pmatrix}
a&0\\ c &d
\end{pmatrix}$, where $a,c,d\in \mathbb{F}$, $ad \neq 0$,
\item $\begin{pmatrix}
a &-\beta_1d\\ 0 &d
\end{pmatrix}$, where $\beta_1,a,d\in \mathbb{F}$, $ad \neq 0$,
\item $\begin{pmatrix}
a &-\beta_1d\\ -\beta_1d &d
\end{pmatrix}\simeq \begin{pmatrix}
ap^2-\beta_1^2d(p-1)^2 &-\beta_1d\\ -\beta_1d &d
\end{pmatrix}$, where $\beta_1,a,d,p\in \mathbb{F}$, $ad-\beta_1^2d^2 \neq 0$, $p \neq 0$,
\item $\begin{pmatrix}
a &1\\ c &0
\end{pmatrix}$, where $a,c\in \mathbb{F}$, $c \neq 0$.
\end{itemize}

Let $G_{7}=\left\{\begin{pmatrix}
 1&0\\1+2t&t
 \end{pmatrix}:\ t\in \mathbb{F} , t\neq0\right\}$($Char(\mathbb{F})=3$). Then one has
$$g^t\begin{pmatrix}
a&b\\ c &d
\end{pmatrix}g=\\ \begin{pmatrix}
a+c(1+2t)+(1+2t)(b+d(1+2t)) &t(b+d(1+2t))\\ tc+td(1+2t) &t^2d
\end{pmatrix}$$ and the following canonical forms occur:
\begin{itemize}
\item $ \begin{pmatrix}
a&0\\ c&d
\end{pmatrix}$, where $a,c,d\in \mathbb{F}$, $ad\neq 0$,\
\item $\begin{pmatrix}
a &-d\\ 0 &d
\end{pmatrix}$, where  $a,d\in \mathbb{F}$, $ ad\neq 0$,\
\item $ \begin{pmatrix}
a &-d\\ -d &d
\end{pmatrix}\simeq  \begin{pmatrix}
a-d+t^2d &-t^2d\\ -t^2d &t^2d
\end{pmatrix}$, where  $a,d,t\in \mathbb{F}$, $d(a-d)\neq 0$, $t \neq 0$,\
\item $\begin{pmatrix}
a&1\\ c &0
\end{pmatrix})$, where  $a,c\in \mathbb{F}$, $c\neq 0$.
\end{itemize}

Finally, taking $g=\begin{pmatrix}
1&0\\ 0&\pm1
\end{pmatrix}\in G_8 $ we get
$g^t\begin{pmatrix}
a &b\\ c&d
\end{pmatrix}g=\begin{pmatrix}
a &\pm b \\ \pm c &d
\end{pmatrix}$ and only the canonical form appears:
\begin{itemize}
\item $\begin{pmatrix}
a &b\\ c &d
\end{pmatrix}\simeq \begin{pmatrix}
a &-b\\ -c&d
\end{pmatrix}$, where $ad-bc\neq 0$.
\end{itemize}

Now we turn to pairs $(\mathbb{A},\sigma)$, where $\mathbb{A}$ is a two-dimensional associative algebra, $\sigma:\mathbb{A}\times \mathbb{A}\longrightarrow \mathbb{F}$ is a non-degenerate bilinear form, up to isomorphism. Taking into account the canonical forms of $\sigma$ along with Theorems \ref{Char0}-\ref{AutChar3} we state the following results.
\begin{lemma} \label{L1} The representatives of isomorphism classes of pairs $(\mathbb{A}, \sigma)$, where $\mathbb{A}$ is a two dimensional associative algebra over a field $\mathbb{F}$ $(Char(\mathbb{F})\neq 2,3)$, $\sigma: \mathbb{A}\times \mathbb{A}\rightarrow \mathbb{F}$ is a non-degenerate form, are given as follows:
\begin{enumerate}
\item $\left(A_3(1,0,0), \begin{pmatrix}
a &1\\ c &d
\end{pmatrix}\right)$, where $a,c,d\in \mathbb{F}$, $ad-c\neq 0$,\\
\item $\left(A_3(1,0,0), \begin{pmatrix}
a &0\\ 1 &d
\end{pmatrix}\right)$, where $a,d\in \mathbb{F}$, $ad\neq 0$,\\
\item $\left(A_3(1,0,0), \begin{pmatrix}
a &0\\ 0 &d
\end{pmatrix}\right)\simeq \left(A_3(1,0,0), \begin{pmatrix}
a &0\\ 0 &t^2d
\end{pmatrix}\right)$, where $a,d, t\in \mathbb{F}$, $ ad\neq 0 , t\neq 0$,\\
\item $\left(A_3(1,0,1), \begin{pmatrix}
a&0\\ 1 &d
\end{pmatrix}\right)$, where $a,d\in \mathbb{F}$, $ad\neq 0$,\\
\item $\left(A_3(1,0,1), \begin{pmatrix}
a &0\\ 0 &d
\end{pmatrix}\right)\simeq \left(A_3(1,0,1), \begin{pmatrix}
a &0\\ 0 &t^2d
\end{pmatrix}\right)$, where  $a,d, t\in \mathbb{F}$, $ ad\neq 0 , t\neq 0$,\\
\item $\left(A_3(1,0,1), \begin{pmatrix}
0 &1\\ c &0
\end{pmatrix}\right)$, where $c\in \mathbb{F}$, $c\neq 0, c +1\neq 0$,\\
\item $\left(A_3(1,0,1), \begin{pmatrix}
a &1\\ -1 &0
\end{pmatrix}\right)$, where $a\in \mathbb{F}$,\\
\item $\left(A_3(\frac{1}{2},0,0), \begin{pmatrix}
a&0\\ 1 &d
\end{pmatrix}\right)$, where $a,d\in \mathbb{F}$, $ad\neq 0$,\\
\item $\left(A_3(\frac{1}{2},0,0), \begin{pmatrix}
a &0\\ 0 &d
\end{pmatrix}\right)\simeq \left(A_3(\frac{1}{2},0,0), \begin{pmatrix}
a &0\\ 0 &t^2d
\end{pmatrix}\right)$, where  $a,d, t\in \mathbb{F}$, $ ad\neq 0 , t\neq 0$,\\
\item $\left(A_3(\frac{1}{2},0,0), \begin{pmatrix}
0 &1\\ c &0
\end{pmatrix}\right)$, where $c\in \mathbb{F}$, $c\neq 0, c +1\neq 0$,\\
\item $\left(A_3(\frac{1}{2},0,0), \begin{pmatrix}
a &1\\ -1 &0
\end{pmatrix}\right)$, where $a\in \mathbb{F}$,\\
\item $\left(A_3(\frac{1}{2},0,\frac{1}{2}), \begin{pmatrix}
a &1\\ c &d
\end{pmatrix}\right)$, where $a,c,d\in \mathbb{F}$, $ad-c\neq 0$,\\
\item $\left(A_3(\frac{1}{2},0,\frac{1}{2}), \begin{pmatrix}
a &0\\ 1 &d
\end{pmatrix}\right)$, where $a,d\in \mathbb{F}$, $ad\neq 0$,\\
\item $\left(A_3(\frac{1}{2},0,\frac{1}{2}), \begin{pmatrix}
a &0\\ 0 &d
\end{pmatrix}\right)\simeq \left(A_3(\frac{1}{2},0,\frac{1}{2}), \begin{pmatrix}
a &0\\ 0 &t^2d
\end{pmatrix}\right)$, where $a,d, t\in \mathbb{F}$, $ ad\neq 0 , t\neq 0$,\\
\item $\left(A_3(\frac{1}{2},\alpha_4,\frac{1}{2}), \begin{pmatrix}
a &b\\ c &d
\end{pmatrix}\right)\simeq \left(A_3(\frac{1}{2},\alpha_4,\frac{1}{2}), \begin{pmatrix}
a &-b\\ -c &d
\end{pmatrix}\right)$,\\ where $\alpha_4,a,b, c,d\in \mathbb{F}$, $\alpha_4\neq 0,ad-bc\neq 0$,\\
\item $\left(A_{13}, \begin{pmatrix}
	a &0\\ c&d
	\end{pmatrix}\right)\simeq \left(A_{13},\begin{pmatrix}
	p^2a &0\\ p^3c &p^4d
	\end{pmatrix}\right), $ where $a,c,d,p\in \mathbb{F}$, $ad\neq 0, p\neq 0$,\\
\item $\left(A_{13}, \begin{pmatrix}
	0 &b\\ c&0
	\end{pmatrix}\right)\simeq \left(A_{13},\begin{pmatrix}
	0 &p^3b\\ p^3c &0
	\end{pmatrix}\right) $, where $b,c,p\in \mathbb{F}$, $bc\neq 0, p\neq 0$,\\
\item $\left(A_{13}, \begin{pmatrix}
	a&b\\ -b &0
	\end{pmatrix}\right)\simeq \left(A_{13},\begin{pmatrix}
	p^2a &p^3b\\ -p^3b &0
	\end{pmatrix}\right)$, where $a,b,p\in \mathbb{F}$, $b^2 \neq 0, p\neq 0$.
\end{enumerate}

\end{lemma}

\begin{lemma} \label{L2} The representatives of isomorphism classes of pairs $(\mathbb{A}, \sigma)$, where $\mathbb{A}$ is a two dimensional associative algebra over a field $\mathbb{F}$ $(Char(\mathbb{F})=2)$, $\sigma: \mathbb{A}\times \mathbb{A}\rightarrow \mathbb{F}$ is a non-degenerate form, are given as follows:
\begin{enumerate}	
	\item $\left(A_{12,2}, \begin{pmatrix}
	a &0\\ c&d
	\end{pmatrix}\right)\simeq \left(A_{12,2},\begin{pmatrix}
	p^2a &0\\ p^3c &p^4d
	\end{pmatrix}\right), $ where $a,c,d,p\in \mathbb{F}$, $ad\neq 0, p\neq 0$,\\
	\item $\left(A_{12,2}, \begin{pmatrix}
	0 &b\\ c &0
	\end{pmatrix}\right)\simeq \left(A_{12,2},\begin{pmatrix}
	0 &p^3b\\ p^3c &0
	\end{pmatrix}\right) $, where $b,c,p\in \mathbb{F}$, $bc\neq 0, p\neq 0$,\\
	\item $\left(A_{12,2}, \begin{pmatrix}
	a &b\\ -b &0
	\end{pmatrix}\right)\simeq \left(A_{12,2},\begin{pmatrix}
	p^2a &p^3b\\ -p^3b &0
	\end{pmatrix}\right)$, where $a,b,p\in \mathbb{F}$, $b^2 \neq 0, p\neq 0$,\\
	\item $\left(A_{11,2}(\beta_1), \begin{pmatrix}
	a&0\\ c &d
	\end{pmatrix}\right)$, where  $\beta_1,a,c,d\in \mathbb{F}$, $ad\neq 0$,\\
	\item $\left(A_{11,2}(\beta_1), \begin{pmatrix}
	a &-\beta_1d\\ 0 &d
	\end{pmatrix}\right)$, where  $\beta_1,a,d\in \mathbb{F}$, $ad \neq 0,$\\
    \item $\left(A_{11,2}(\beta_1), \begin{pmatrix}
	a &-\beta_1d\\ -\beta_1d &d
	\end{pmatrix}\right)\simeq \left(A_{11,2}(\beta_1), \begin{pmatrix}
	ap^2-\beta_1^2d(p-1)^2 &-\beta_1d \\ -\beta_1d &d
	\end{pmatrix}\right)$,\\ where $\beta_1,a,d,p\in \mathbb{F}$,  $ad-\beta_1^2d^2 \neq 0 , p \neq 0$,\\
	\item $\left(A_{11,2}(\beta_1), \begin{pmatrix}
	a &1\\ c &0
	\end{pmatrix}\right)$, where $\beta_1,a,c\in \mathbb{F}$, $c \neq 0$,\\
	\item $\left(A_{6,2}(1,0), \begin{pmatrix}
	a&b\\ c &d
	\end{pmatrix}\right)$, where $a,b,c, d\in \mathbb{F}$,  $ad-bc\neq 0$,\\
	\item $\left(A_{4,2}(1,0,0), \begin{pmatrix}
	a&0\\ 1 &d
	\end{pmatrix}\right)$, where $a,d\in \mathbb{F}$, $ad\neq 0$,\\
	\item $\left(A_{4,2}(1,0,0), \begin{pmatrix}
	a &0\\ 0 &d
	\end{pmatrix}\right)\simeq \left(A_{4,2}(1,0,0), \begin{pmatrix}
	p^2a &0\\ 0 &d
	\end{pmatrix}\right)$, where $a,d,p \in \mathbb{F}$, $ad\neq 0 , p\neq 0$,\\
	\item $\left(A_{4,2}(1,0,0), \begin{pmatrix}
	0 &1\\ c &0
	\end{pmatrix}\right)$, where $c\in \mathbb{F}$,  $c\neq 0, c +1\neq 0$,\\
	\item $\left(A_{4,2}(1,0,0), \begin{pmatrix}
	a &1\\ -1 &0
	\end{pmatrix}\right)$, where $a\in \mathbb{F}$,\\
	\item $\left(A_{4,2}(1,\beta_1,0), \begin{pmatrix}
	a&0\\ c &d
	\end{pmatrix}\right)$, where  $\beta_1,a,c, d\in \mathbb{F}$, $\beta_1 \neq 0$, $ad\neq 0$,\\
	\item $\left(A_{4,2}(1,\beta_1,0), \begin{pmatrix}
	0&b\\ ca &0
	\end{pmatrix}\right)$, where  $\beta_1,b, c\in \mathbb{F}$, $\beta_1 \neq 0$, $ bc\neq 0 , b+c\neq 0$,\\
	\item $\left(A_{4,2}(1,\beta_1,0), \begin{pmatrix}
	a &b\\ -b &0
	\end{pmatrix}\right)$, where  $\beta_1,a,b\in \mathbb{F}$, $\beta_1 \neq 0$, $b\neq 0$,\\
	\item $\left(A_{3,2}(1,0,0), \begin{pmatrix}
	a &1\\ c &d
	\end{pmatrix}\right)$, where  $a,c, d\in \mathbb{F}$, $ad-c\neq 0$,\\
	\item $\left(A_{3,2}(1,0,0), \begin{pmatrix}
	a &0\\ 1 &d
	\end{pmatrix}\right)$, where $a, d\in \mathbb{F}$,  $ad\neq 0$,\\
	\item $\left(A_{3,2}(1,0,0), \begin{pmatrix}
	a &0\\ 0 &d
	\end{pmatrix}\right)\simeq \left(A_{3,2}(1,0,0), \begin{pmatrix}
	a &0\\ 0 &t^2d
	\end{pmatrix}\right)$, where $a, d,t \in \mathbb{F}$,  $ ad\neq 0 , t\neq 0$,\\
	\item $\left(A_{3,2}(1,0,1), \begin{pmatrix}
	a&0\\ 1 &d
	\end{pmatrix}\right)$, where $a, d\in \mathbb{F}$,  $ad\neq 0$,\\
	\item $\left(A_{3,2}(1,0,1), \begin{pmatrix}
	a &0\\ 0 &d
	\end{pmatrix}\right)\simeq \left(A_{3,2}(1,0,1), \begin{pmatrix}
	a&0\\ 0 &t^2d
	\end{pmatrix}\right)$, where $a,d,t\in \mathbb{F}$, $ ad\neq 0 , t \neq 0$,\\
	\item $\left(A_{3,2}(1,0,1), \begin{pmatrix}
	0 &1\\ c &0
	\end{pmatrix}\right)$, where $c\in \mathbb{F}$, $c\neq 0, c +1\neq 0$,\\
	\item $\left(A_{3,2}(1,0,1), \begin{pmatrix}
	a &1\\ -1 &0
	\end{pmatrix}\right)$, where $a\in \mathbb{F}$.
\end{enumerate}

\end{lemma}

\begin{lemma} \label{L3} The representatives of isomorphism classes of pairs $(\mathbb{A}, \sigma)$, where $\mathbb{A}$ is a two dimensional associative algebra over a field $\mathbb{F}$ $(Char(\mathbb{F})=3)$, $\sigma: \mathbb{A}\times \mathbb{A}\rightarrow \mathbb{F}$ is a non-degenerate form, are given as follows:
\begin{enumerate}	
	\item $\left(A_{13,3}, \begin{pmatrix}
a &0\\ c &d
\end{pmatrix}\right)\simeq \left(A_{13,3},\begin{pmatrix}
p^2a &0\\ p^3c &p^4d
\end{pmatrix}\right) , $ where $a,c,d,p\in \mathbb{F}$, $ad\neq 0 , p\neq 0$,\\
\item $\left(A_{13,3}, \begin{pmatrix}
0 &b\\ c &0
\end{pmatrix}\right)\simeq\left(A_{13,3}, \begin{pmatrix}
	0 &p^3b\\ p^3c &0
\end{pmatrix}\right) $, where $b,c,p\in \mathbb{F}$, $b+c\neq 0, bc\neq 0 , p\neq 0$,\\
\item $\left(A_{13,3}, \begin{pmatrix}
a &b\\ -b &0
\end{pmatrix}\right)\simeq \left(A_{13,3},\begin{pmatrix}
p^2a &p^3b\\ -p^3b &0
\end{pmatrix}\right)$, where $a,b,p\in \mathbb{F}$, $b^2\neq 0 , p\neq 0$,\\
\item $\left(A_{3,3}(1,0,1), \begin{pmatrix}
a&0\\ 1 &d
\end{pmatrix}\right)$, where $a,d\in \mathbb{F}$, $ad\neq 0$,\\
\item $\left(A_{3,3}(1,0,1), \begin{pmatrix}
a &0\\ 0 &d
\end{pmatrix}\right)\simeq \left(A_{3,3}(1,0,1), \begin{pmatrix}
a &0\\ 0 &t^2d
\end{pmatrix}\right)$, where  $a,d, t\in \mathbb{F}$, $ ad\neq 0 , t\neq 0$,\\
\item $\left(A_{3,3}(1,0,1), \begin{pmatrix}
0 &1\\ c &0
\end{pmatrix}\right)$, where $c\in \mathbb{F}$, $c\neq 0, c+1\neq 0$,\\
\item $\left(A_{3,3}(1,0,1), \begin{pmatrix}
a &1\\ -1 &0
\end{pmatrix}\right)$, where $a\in \mathbb{F}$,\\
\item $\left(A_{3,3}(1,0,0), \begin{pmatrix}
a &1\\ c &d
\end{pmatrix}\right)$, where $a,c,d\in \mathbb{F}$, $ad-c\neq 0$,\\
\item $\left(A_{3,3}(1,0,0), \begin{pmatrix}
a &0\\ 1 &d
\end{pmatrix}\right)$, where $a,d\in \mathbb{F}$, $ad\neq 0$,\\
\item $\left(A_{3,3}(1,0,0), \begin{pmatrix}
a &0\\ 0 &d
\end{pmatrix}\right)\simeq \left(A_{3,3}(1,0,0), \begin{pmatrix}
a &0\\ 0 &t^2d
\end{pmatrix}\right)$, where $a,d, t\in \mathbb{F}$, $ ad\neq 0 , t\neq 0$,\\
\item $\left(A_{3,3}(2,0,0), \begin{pmatrix}
a&0\\ c&d
\end{pmatrix}\right)$, where $a,c,d\in \mathbb{F}$, $ad\neq 0$,\\
\item $\left(A_{3,3}(2,0,0), \begin{pmatrix}
a &-d\\ 0 &d
\end{pmatrix}\right)$, where $a,d\in \mathbb{F}$, $ ad\neq 0$,\\
\item $\left(A_{3,3}(2,0,0), \begin{pmatrix}
a &-d\\ -d&d
\end{pmatrix}\right)\simeq \left(A_{3,3}(2,0,0), \begin{pmatrix}
a-d+t^2d &-t^2d\\ -t^2d &t^2d
\end{pmatrix}\right)$, where $a,d, t\in \mathbb{F}$,

\hfill $d(a-d)\neq 0$, $t \neq 0$,
\item $\left(A_{3,3}(2,0,0), \begin{pmatrix}
a&1\\ c &0
\end{pmatrix}\right)$, where $a,c\in\mathbb{F}$, $c\neq 0$,\\
\item $\left(A_{3,3}(2,0,2), \begin{pmatrix}
a&b\\ c &d
\end{pmatrix}\right)$, where $a,b,c, d \in \mathbb{F}$ , $ad-bc\neq 0$,\\
\item $\left(A_{3,3}(2,\alpha_4,2), \begin{pmatrix}
a&0\\ c &d
\end{pmatrix}\right)$, where  $\alpha_4,a,c, d\in \mathbb{F}$, $\alpha_4 \neq 0$, $ad\neq 0$,\\
\item $\left(A_{3,3}(2,\alpha_4,2), \begin{pmatrix}
0&b\\ c &0
\end{pmatrix}\right)$, where  $\alpha_4,b, c\in \mathbb{F}$, $\alpha_4 \neq 0$, $ bc\neq 0 , b+c\neq 0$,\\
\item $\left(A_{3,3}(2,\alpha_4,2), \begin{pmatrix}
a &b\\ -b &0
\end{pmatrix}\right)$, where  $\alpha_4,a,b\in \mathbb{F}$, $\alpha_4 \neq 0$, $b\neq 0$.\\

\end{enumerate}

\end{lemma}

\section{Classification of two-dimensional Frobenius algebras}

Now we specify those pairs from the lammas above which are Frobenius algebras.

\begin{theor} \emph{}

\begin{itemize}
\item If $Char(\mathbb{F})\neq 2,3$ then any two-dimensional Frobenius algebra over $\mathbb{F}$ is isomorphic to only one of the following such algebras:
\begin{itemize}	
	\item[$\ast$] $\left(A_3(1,0,0), \begin{pmatrix}
	a &0\\ 0 &d
	\end{pmatrix}\right)\simeq \left(A_3(1,0,0), \begin{pmatrix}
	a &0\\ 0 &t^2d
	\end{pmatrix}\right)$, where $a,d, t\in \mathbb{F}, t \neq 0, ad\neq 0$,
	\item[$\ast$]$\left(A_3\left(\frac{1}{2},0,\frac{1}{2}\right), \begin{pmatrix}
	a &1\\ 1 &0
	\end{pmatrix}\right)$, where $a\in \mathbb{F}$,
	\item[$\ast$] $\left(A_3\left(\frac{1}{2},\alpha_4,\frac{1}{2}\right), \begin{pmatrix}
	a &b\\ b&2\alpha_4a
	\end{pmatrix}\right)\simeq \left(A_3\left(\frac{1}{2},\alpha_4,\frac{1}{2}\right), \begin{pmatrix}
	a &-b\\ -b&2\alpha_4a
	\end{pmatrix}\right)$, where $\alpha_4,a,b\in \mathbb{F}$,  $2\alpha_4a^2-b^2\neq 0 ,\alpha_4 \neq 0 $.\\	
	\item[$\ast$] $\left(A_{13}, \begin{pmatrix}
	0 &b\\ b &0
	\end{pmatrix}\right)\simeq \left(A_{13},\begin{pmatrix}
	0 &p^3b\\ p^3b &0
	\end{pmatrix}\right) $, where $b,p\in \mathbb{F}$, $p\neq 0,b\neq 0$.
\end{itemize}
\item If $Char(\mathbb{F})=2$ then any two-dimensional Frobenius algebra over $\mathbb{F}$ is isomorphic to only one of the following such algebras:	
	\begin{itemize}
\item[$\ast$] $\left(A_{3,2}(1,0,0), \begin{pmatrix}
	a &0\\ 0 &d
	\end{pmatrix}\right)\simeq \left(A_{3,2}(1,0,0), \begin{pmatrix}
	a &0\\ 0 &t^2d
	\end{pmatrix}\right)$, where $a,d,t\in \mathbb{F},t\neq 0, ad\neq 0$,
		\item[$\ast$] $\left(A_{4,2}(1,0,0), \begin{pmatrix}
		1 &1\\ -1 &0
		\end{pmatrix}\right)$,
		\item[$\ast$] $\left(A_{4,2}(1,\beta_1,0), \begin{pmatrix}
		\beta_1d &0\\ 0 &d
		\end{pmatrix}\right)$, where $\beta_1,d\in \mathbb{F}, \beta_1d\neq 0$,
		\item[$\ast$] $\left(A_{4,2}(1,\beta_1,0), \begin{pmatrix}
		b &b\\ -b &0
		\end{pmatrix}\right)$, where $\beta_1,b\in \mathbb{F},\beta_1 \neq 0, b\neq 0$,
		\item[$\ast$] $\left(A_{11,2}(\beta_1), \begin{pmatrix}
	\beta_1d &0\\ 0 & d
	\end{pmatrix}\right)$, where $\beta_1,d\in \mathbb{F}$, $\beta_1d^2 \neq 0$,	
        \item[$\ast$] $\left(A_{11,2}(\beta_1), \begin{pmatrix}
	\beta_1d &-\beta_1d\\ -\beta_1d &d
	\end{pmatrix}\right)\simeq \left(A_{11,2}(\beta_1), \begin{pmatrix}
	\beta_1dp^2-\beta_1^2d(p-1)^2 &-\beta_1d\\ -\beta_1d &d
	\end{pmatrix}\right)$,\\ where $\beta_1,d,p\in \mathbb{F}$,  $p \neq 0, \beta_1d^2(1-\beta_1) \neq 0 $,
	\item[$\ast$] $\left(A_{11,2}(\beta_1), \begin{pmatrix}
	0 &1\\ 1 &0
	\end{pmatrix}\right)$, where $\beta_1\in \mathbb{F}$,
			\item[$\ast$] $\left(A_{12,2}, \begin{pmatrix}
		0 &b\\ b &0
		\end{pmatrix}\right)\simeq \left(A_{12,2},\begin{pmatrix}
		0 &p^3b\\ p^3b &0
		\end{pmatrix}\right) $, where $b,p\in \mathbb{F}$, $p\neq 0,b\neq 0$.
		\end{itemize}
\item If $Char(\mathbb{F})=3$ then any two-dimensional Frobenius algebra over $\mathbb{F}$ is isomorphic to only one of the following such algebras:
\begin{itemize}	
\item[$\ast$] $\left(A_{3,3}(1,0,0), \begin{pmatrix}
	a &0\\ 0 &d
	\end{pmatrix}\right)\simeq \left(A_{3,3}(1,0,0), \begin{pmatrix}
	a &0\\ 0 &t^2d
	\end{pmatrix}\right)$, where $a,d,t\in \mathbb{F},t\neq 0, ad\neq 0$,
	\item[$\ast$] $\left(A_{3,3}(2,0,2), \begin{pmatrix}
	a&b\\ b &0
	\end{pmatrix}\right)$, where $a,b\in \mathbb{F}, b\neq 0$,
	\item[$\ast$] $\left(A_{3,3}(2,\alpha_4,2), \begin{pmatrix}
	a&0\\ 0 &2\alpha_4a
	\end{pmatrix}\right)$, where $\alpha_4,a\in \mathbb{F}, \alpha_4a\neq 0$,
    \item[$\ast$] $\left(A_{3,3}(2,\alpha_4,2), \begin{pmatrix}
	0&b\\ b &0
	\end{pmatrix}\right)$, where $\alpha_4, b\in \mathbb{F}, \alpha_4\neq 0, b\neq 0$,
		\item[$\ast$] $\left(A_{13,3}, \begin{pmatrix}
		0 &b\\ b&0
		\end{pmatrix}\right)\simeq \left(A_{13,3},\begin{pmatrix}
		0 &p^3b\\ p^3b &0
		\end{pmatrix}\right) $, where $b,p\in \mathbb{F},p\neq 0,$ $b\neq 0$.
		
\end{itemize}
\end{itemize}	
	
\end{theor}

\begin{proof}
Let now check the condition $\sigma(\mathrm{x y,z})=\sigma(\mathrm{x,yz})$ for the pairs appeared in Lemma \ref{L1} ($Char \mathbb{F} \neq 2,3$), Lemma \ref{L2} ($Char \mathbb{F} = 2$) and Lemma \ref{L3} ($Char \mathbb{F} =3$), i.e., find the solutions to the system of equations (\ref{Form2}).

\underline{Let $Char \mathbb{F} \neq 2,3$.}
In this case the system of equations (\ref{Form2}) is consistent only for pairs $(3)$,  $(12)$, $(15)$ and $(17)$. The solutions to the system are as follows:

For $(3)$ $\left(A_3(1,0,0), \begin{pmatrix}
a &0\\ 0 &d
\end{pmatrix}\right)\simeq \left(A_3(1,0,0), \begin{pmatrix}
a &0\\ 0 &t^2d
\end{pmatrix}\right)$, where $a, d, t\in \mathbb{F}, ad\neq 0, t\neq 0$
the system becomes an identity, therefore, the pair $(3)$ is a Frobenius algebra.

%
%


Consider $(12)$ $\left(A_3(\frac{1}{2},0,\frac{1}{2}), \begin{pmatrix}
a &1\\ c &d
\end{pmatrix}\right)$, where $a,c,d\in \mathbb{F}$, $ad-c\neq 0$. As a solution the system (\ref{Form2}) we get $c=1$ and $d=0$ and the corresponding Frobenius algebras are $$\left(A_3\left(\frac{1}{2},0,\frac{1}{2}\right), \begin{pmatrix}
a &1\\ 1 &0
\end{pmatrix}\right), \ \mbox{where} \  a\in \mathbb{F}.$$

Among the pairs $(15)$ $\left(A_3\left(\frac{1}{2},\alpha_4,\frac{1}{2}\right), \begin{pmatrix}
a &b\\ c &d
\end{pmatrix}\right)\simeq \left(A_3\left(\frac{1}{2},\alpha_4,\frac{1}{2}\right), \begin{pmatrix}
a &-b\\ -c &d
\end{pmatrix}\right)$, where  $a,b,c,d,\alpha_4\in \mathbb{F}$, $ad-bc\neq 0$, $\alpha_4 \neq 0$ the algebras satisfying $ c=b,$ and $ d=2\alpha_4a$ are Frobenius. Thus, we get

%
%
%
%
%

$$\left(A_3\left(\frac{1}{2},\alpha_4,\frac{1}{2}\right), \begin{pmatrix}
a &b\\ b &2\alpha_4a
\end{pmatrix}\right)\simeq \left(A_3\left(\frac{1}{2},\alpha_4,\frac{1}{2}\right), \begin{pmatrix}
a &-b\\ -b &2\alpha_4a
\end{pmatrix}\right),$$ where $a,b,\alpha_4\in \mathbb{F}$,  $2\alpha_4a^2-b^2\neq 0 ,\alpha_4 \neq 0 $.

Considering the system of equations (\ref{Form2}) for $(17)$ $\left(A_{13}, \begin{pmatrix}
0 &b\\ c &0
\end{pmatrix}\right)\simeq \left(A_{13},\begin{pmatrix}
	0 &p^3b\\ p^3c &0
\end{pmatrix}\right) $, where $p\neq 0$ we obtain $c=b$ and hence,
$$\left(A_{13}, \begin{pmatrix}
0 &b\\ b &0
\end{pmatrix}\right)\simeq \left(A_{13},\begin{pmatrix}
	0 &p^3b\\ p^3b &0
\end{pmatrix}\right) \ \mbox{where}\ b,p\in \mathbb{F} \ p\neq 0, \ \mbox{and}\  b^2\neq 0$$ is a Frobenius algebra.

\underline{Let now treat the case of $Char \mathbb{F} = 2$.}
From the pairs $(2)$ $\left(A_{12,2}, \begin{pmatrix}
	0 &b\\ c &0
	\end{pmatrix}\right)\simeq \left(A_{12,2},\begin{pmatrix}
	0 &p^3b\\ p^3c &0
	\end{pmatrix}\right) $, where $b,c,p\in \mathbb{F}$, $bc\neq 0, p\neq 0$ here we generate the following Frobenius algebras
$\left(A_{12,2}, \begin{pmatrix}
	0 &b\\ b &0
	\end{pmatrix}\right)\simeq \left(A_{12,2},\begin{pmatrix}
	0 &p^3b\\ p^3b &0
	\end{pmatrix}\right) $, where $b,p\in \mathbb{F}$, $p\neq 0,b^2\neq 0$ since the solution to the system of equations (\ref{Form2}) in this case is
$c=b.$

%

For $(4)$ $\left(A_{11,2}(\beta_1), \begin{pmatrix}
a & 0\\ c &d
\end{pmatrix}\right)$, where $\beta_1,a,c,d\in \mathbb{F}$, $ad \neq 0$,

the solution to the system of equations (\ref{Form2}) is $ c=0$ and $a=\beta_1d$. Therefore, $\left(A_{11,2}(\beta_1), \begin{pmatrix}
\beta_1d & 0 \\ 0 &d
\end{pmatrix}\right)$ is a Frobenius algebra.

From $(6)$ $\left(A_{11,2}(\beta_1), \begin{pmatrix}
a &-\beta_1d \\ -\beta_1d &d
\end{pmatrix})\simeq (A_{11,2}(\beta_1), \begin{pmatrix}
ap^2-\beta_1^2d(p-1)^2 &-\beta_1d \\ -\beta_1d &d
\end{pmatrix}\right)$, where $\beta_1,a,d,p\in \mathbb{F}$,  $ad-\beta_1^2d^2 \neq 0 , p \neq 0$ subjecting to the system (\ref{Form2}) we get
%
%
%
%
%
$ a=\beta_1d$ and obtain the following Frobenius algebra $$\left(A_{11,2}(\beta_1), \begin{pmatrix}
\beta_1d &-\beta_1d \\ -\beta_1d &d
\end{pmatrix}\right)\simeq \left(A_{11,2}(\beta_1), \begin{pmatrix}
\beta_1dp^2-\beta_1^2d(p-1)^2 &-\beta_1d \\ -\beta_1d &d
\end{pmatrix}\right),\ \mbox{where}\ \beta_1,d,p\in \mathbb{F}, \beta_1d^2(1-\beta_1)\neq 0, p \neq 0.$$

For $(7)$ $\left(A_{11,2}(\beta_1), \begin{pmatrix}
a &1\\ c &0
\end{pmatrix}\right)$, where $\beta_1,a,c\in \mathbb{F}$,

the solution to the system of equations (\ref{Form2}) is $c=1$ and $ a=0.$ Therefore, $\left(A_{11,2}(\beta_1), \begin{pmatrix}
0 &1\\ 1 &0
\end{pmatrix}\right)$ is a Frobenius algebra.

Considering $(12)$  $\left(A_{4,2}(1,0,0), \begin{pmatrix}
a &1\\ -1 &0
\end{pmatrix}\right)$, where $a\in \mathbb{F}$, we obtain $ a=1,$ i.e.,
 $$\left(A_{4,2}(1,0,0), \begin{pmatrix}
1 &1\\ -1 &0
\end{pmatrix}\right)$$
is a Frobenius algebra here.

Let us consider $(13)$  $\left(A_{4,2}(1,\beta_1,0), \begin{pmatrix}
a &0\\ c &d
\end{pmatrix}\right)$, where $\beta_1 \neq 0$ and $ ad\neq 0$. Then the system of equations (\ref{Form2}) is equivalent to

$$ \left\lbrace\begin{array}{cc}
c=0\\
a-\beta_1d=0\\
\end{array}\right.
$$

Therefore, $\left(A_{4,2}(1,\beta_1,0), \begin{pmatrix}
\beta_1d &0\\ 0 &d
\end{pmatrix}\right)$, where $\beta_1,d\in \mathbb{F}, \beta_1d^2\neq 0$ are Frobenius algebras.

From $(15)$ $\left(A_{4,2}(1,\beta_1,0), \begin{pmatrix}
a &b\\ -b &0
\end{pmatrix}\right)$, where $\beta_1,a,b\in \mathbb{F}$, $\beta_1 \neq 0$ and $b\neq 0$ as a solution to the system (\ref{Form2})
we obtain $b=a$. Hence, $$\left(A_{4,2}(1,\beta_1,0), \begin{pmatrix}
b &b\\ -b&0
\end{pmatrix}\right), \  \mbox{where}\ \beta_1,b\in \mathbb{F},\ \beta_1 \neq0 \ \mbox{and}\ b\neq 0.$$
are Frobenius algebras.

For $(18)$ $\left(A_{3,2}(1,0,0), \begin{pmatrix}
a &0\\ 0 &d
\end{pmatrix}\right)\simeq \left(A_{3,2}(1,0,0), \begin{pmatrix}
a &0\\ 0 &t^2d
\end{pmatrix}\right)$, where $a, d, t\in \mathbb{F}, t\neq 0, ad\neq 0$ the system of equations (\ref{Form2}) is an identity and they are Frobenius algebras.

Note that for all the other classes from Lemma \ref{L2} the system of equations (\ref{Form2}) is inconsistent.


\underline{Let $Char \mathbb{F} =3.$}
We verify the list of pairs from Lemma \ref{L3} to be Frobenius algebras.
From $(2)$  $\left(A_{13,3}, \begin{pmatrix}
0 &b\\ c&0
\end{pmatrix}\right)\simeq \left(\begin{pmatrix}
	0 &p^3b\\ p^3c&0
\end{pmatrix}\right) $, where $bc\neq 0$ as a solution to the system (\ref{Form2}) we get $c=b$ that means $\left(A_{13,3}, \begin{pmatrix}
0 &b\\ b &0
\end{pmatrix}\right) \simeq \left(\begin{pmatrix}
	0 &p^3b\\ p^3b &0
\end{pmatrix}\right) $, where $b, p\in \mathbb{F}, p\neq 0, $ $b^2\neq 0$ are Frobenius algebras.

All the pairs $(10)$ $\left(A_{3,3}(1,0,0), \begin{pmatrix}
a &0\\ 0 &d
\end{pmatrix}\right)\simeq \left(A_{3,3}(1,0,0), \begin{pmatrix}
a &0\\ 0 &t^2d
\end{pmatrix}\right)$, where $0\neq t\in \mathbb{F}, ad\neq 0$ are Frobenius algebras.

It is easy to see that among the pairs $(15)$ $\left(A_{3,3}(2,0,2), \begin{pmatrix}
a&b\\ c &d
\end{pmatrix}\right)$, where $ad-bc\neq 0$ Frobenius algebras are $$\left(A_{3,3}(2,0,2), \begin{pmatrix}
a&b\\ b &0
\end{pmatrix}\right),\ \mbox{where}\ a, b\in \mathbb{F}, \ b^2\neq 0$$

From the pairs $(16)$ $\left(A_{3,3}(2,\alpha_4,2), \begin{pmatrix}
a&0\\ c&d
\end{pmatrix}\right)$, where $ad\neq 0$ and $\alpha_4 \neq 0$, we generate the following Frobenius algebras  $$\left(A_{3,3}(2,\alpha_4,2), \begin{pmatrix}
a&0\\ 0 &2\alpha_4a
\end{pmatrix}\right),\ \mbox{where}\ \alpha_4,\ a\in \mathbb{F},\ 2\alpha_4a^2\neq 0.$$

Finally, considering $(17)$ $\left(A_{3,3}(2,\alpha_4,2), \begin{pmatrix}
0&b\\ c &0
\end{pmatrix}\right)$, where $\alpha_4 \neq 0$, $bc\neq 0, \ b+c\neq 0$,

we obtain the following Frobenius algebras $$\left(A_{3,3}(2,\alpha_4,2), \begin{pmatrix}
0&b\\ b &0
\end{pmatrix}\right),\ \mbox{where}\ \alpha_4,b\in \mathbb{F},\ \alpha_4 \neq 0,\ b\neq 0.$$

For all the other pairs the system (\ref{Form2}) is inconsistent.

\end{proof}

\section{Conclusion}
The study of Frobenius algebras continues to be a vibrant and active area of research. They are deeply connected to multiple areas of mathematics, from algebraic geometry and representation theory to mathematical physics and homological algebra. In this paper we classified all Frobenius algebra structures on two-dimensional vector space over any base field. Further exploration into their generalizations, homological structures, and physical applications remains a rich area for future research.

\section{Conflicts of Interest}
There is no conflict of interest to declare.


\begin{thebibliography}{9}
\bibitem {A} {\sc M. Atiyah}, Topological Quantum Field Theories, {\it Publications math\'{e}matiques de l' I.H.\'{E}.S.}, 1988, 68, 175--186.
\bibitem {LA} {\sc L. Abrams}, Two-dimensional Topological Quantum Field Theories and Frobenius algebras, {\it Journal of Knot Theory and Its Ramifications}, 1996, 5(5), 569--587.
\bibitem {TQFT1} {\sc D. Bar-Natan}, Khovanov's homology for tangles and cobordisms, {\it Geom. Topol.}, 2005, 9(3), 1443--1499, arXiv:math/0410495, doi:10.2140/gt.2005.9.1443.
\bibitem{UB} {\sc Bekbaev U.}, Classification of two-dimensional algebras over any base field, 2023, {\it AIP Conference Proceedings}, 2880, 030001, https://doi.org/10.1063/5.0165726.
\bibitem {BrauerNesbitt}  {\sc R. Brauer, C. Nesbitt}, On the regular representations of algebras, {\it Proc. Natl. Acad. Sci. USA}, 1937, 23(4), 236--240, doi:10.1073/pnas.23.4.236.
\bibitem {CIM} {\sc S. Caenepeel, B. Ion, G. Militaru},  The structure of Frobenius algebras and separable algebras, {\it Journal of K-Theory}, 2000, 19(4), 365--402.
DOI: 10.1023/A:1007849203555
\bibitem {CR} {\sc C.W. Curtis, I. Reiner}. {\it Representation theory of finite groups and associative algebras}, 1962,   John Wiley $\&$ Sons, New York $\cdot$ London $\cdot$ Sydney, 689 pp.
\bibitem {Dieudonné(1958)} {\sc J. Dieudonn\'{e}}, Remarks on quasi-Frobenius rings, {\it Illinois Journal of Mathematics}, 1958, 2(3), 346--354, doi:10.1215/ijm/1255454538.
\bibitem {paul} {\sc P. Grosskopf, J. Vercruysse}, {\it The Hopf category of Frobenius algebras},  arXiv:2406.18499v1 [math.QA] 26 Jun 2024.
\bibitem {Lauda0} {\sc A.D. Lauda}, {\it Frobenius algebras and planar open string
topological field theories},  arXiv:math/0508349v1 [math.QA].
\bibitem {Lauda} {\sc A.D. Lauda, H. Pfeiffer}, Open–closed strings: Two-dimensional extended TQFTs and Frobenius algebras, {\it Topology and its Applications}, 2008, 155, 623--666.
\bibitem {Nakayama1} {\sc T. Nakayama}, On Frobeniusean algebras I, {\it Ann. of Math.}, 1939, 40, 611--633,
\bibitem {Nakayama2} {\sc T. Nakayama}, On Frobeniusean algebras II, {\it Ann. of Math.}, 1941, 42, 1--21.
\bibitem {DR} {\sc D.E Radford}, The structure of Hopf algebras with a projection, {\it Journal of Algebra}, 1985, 92(2), 322-347.
\bibitem {Rakhimov} {\sc I.S.Rakhimov}, Algebraic Structures on Two-Dimensional Vector Space Over Any Basic Field, {\it Malaysian Journal of Mathematical Sciences}, 2024, 18(2), 227--257, https://doi.org/10.47836/mjms.18.2.02
\bibitem {TQFT2} {\sc  P. Turner}, {\it Five Lectures on Khovanov Homology}, 2006, arXiv:math/0606464.
\end{thebibliography}
 	\end{document}